\documentclass[a4paper,12pt,oneside]{article}
\usepackage{amsmath, amssymb, graphicx, amsthm, mathrsfs}
\usepackage{mathptmx}
\usepackage{mathtools} 
\usepackage{geometry}
\usepackage{enumitem}  
\usepackage{hyperref}
\usepackage{tikz}
\usepackage{pgfplots}
\pgfplotsset{compat=1.18}
\newtheorem{theorem}{Theorem}[section]
\newtheorem{definition}[theorem]{Definition}
\newtheorem{lemma}[theorem]{Lemma}
\newtheorem{proposition}[theorem]{Proposition}
\newtheorem{corollary}[theorem]{Corollary}

\newtheorem{remark}[theorem]{Remark}
\newtheorem{example}[theorem]{Example}
\numberwithin{equation}{section}

\title{Perimetric Contractions and Their Iterates in Complete $b$-Metric Spaces}
\author{
 Mujahid Abbas\textsuperscript{1},  Alemayehu G. Negashs\textsuperscript{2,*}, Meaza F. Bogale \textsuperscript{3}\\
\small \textsuperscript{1}Department of Mechanical Engineering Sciences, Faculty of Engineering and the Built Environment,\\ \small Doornfontein Campus, University of Johannesburg, South Africa, \\ \small and Department of Medical Research, China Medical University,\\ \small Taichung 404, Taiwan. \\
\small \textsuperscript{1}Email: \texttt{ mujahid@uj.ac.za} \\
\small \textsuperscript{2,3}Department of Mathematics, Hampton University, USA\\
$^*$\small  \textit{Corresponding author} \\
\small \textsuperscript{\dag}Email: \texttt{alemayehu.negash@hamptonu.edu}\\
\small \textsuperscript{\ddag}Email: \texttt{meaza.bogale@hamptonu.edu}
}
\date{}

\begin{document}

\maketitle

\begin{abstract}
In this paper, the structural and operator-theoretic properties of contracting perimeters of triangles mapping (CPTM) within the generalized topological framework of complete $b$-metric spaces with coefficient $s \geq 1$, is systematically investigated. Extending recent foundational advancements from classical metric spaces, we explore the architectural interplay between multi-point perimetric constraints and path-wise orbital stability under two distinct structural framework. First, assuming the minimal exclusion of periodic orbits of prime period two, we prove that the higher-order iterates $f^{n}$ of an CPTM behave as graphic contractions for all indices satisfying the condition $sq^{n} < 1$. This classifies the operator as a weakly Picard operator and yields a unified existence and cardinality theorem establishing that the fixed-point set satisfies $1 \leq |\mathrm{Fix}(f)| \leq 2$.\\ Second, in the alternative configuration where the operator does possess a periodic orbit of prime period two, we resolve a significant structural gap under the parameter condition $sq^{2} < 1$. We demonstrate that the higher even iterates $f^{2n}$ collapse into continuous graphic contractions, proving that the mapping possesses exactly two periodic points which form a single, isolated 2-cycle. Throughout our proofs, we rigorously navigate the analytical challenges arising from the potential simultaneous non-continuity of the $b$-metric function by relying strictly on sequential tracking inequalities. Finally, we present concrete analytical examples, including a shift map on a discrete metric space, to show that the class of CPTM is strictly larger than the class of graphic contractions, thereby demonstrating the sharpness and optimality of the obtained parameter conditions.
\end{abstract}

\maketitle
\section{Introduction}
Fixed-point theory stands as a cornerstone of modern nonlinear analysis providing an appropriate mathematical tools to establish the existence, uniqueness, and approximation of solutions to differential, integral, and functional equations. The classical paradigm, anchored by the foundational Banach Contraction Principle relies on a complete metric space $(X,d)$ and a self-mapping $f:X\to X$ satisfying a pairwise contractive condition:
\begin{equation}
d(f(x),f(y))\leq \alpha d(x,y)\quad \forall x,y\in X,
\label{eq:banach}
\end{equation}
where $\alpha \in [0,1)$. Over the past several decades, research in this domain has expanded along two main, interconnected trajectories: the structural relaxation of the underlying geometric and topological properties of the space, and the formulation of multi-point contractive conditions that capture complex multi-variable geometric invariants.

A highly prominent structural generalization of classical metric spaces was introduced independently by Bakhtin \cite{bakhtin1989} and Czerwik \cite{czerwik1993} through the concept of a $b$-metric space. By adjusting the classical triangle inequality with a topological scaling parameter $s \geq 1$ such that
\begin{equation}
d(x,z)\leq s[d(x,y) + d(y,z)]\quad \forall x,y,z\in X,
\label{eq:bmetric_tri}
\end{equation}
 successfully encompasses a wide range of functionally significant spaces. Notable examples include the Lebesgue spaces $L^{p}(0,1)$ and the sequence spaces $\ell^{p}$ for $0< p< 1$, both of which fail the standard triangle inequality but naturally satisfy the $b$-metric axioms. A comprehensive overview of these early architectural developments and related topological concepts can be found in the definitive survey by Berinde and Păcurar \cite{berinde2024}.

Despite their broad applicability, $b$-metric spaces present formidable analytical hurdles. Most notably, the distance function $d(\cdot,\cdot)$ is generally not continuous in both variables simultaneously when $s > 1$ and closed balls are not necessarily closed sets. This loss of simultaneous continuity complicates limiting-passing across perimetric inequalities and requires rigorous bounding mechanisms, such as the sequence tracking techniques pioneered by Miculescu and Mihail \cite{miculescu2017} and Suzuki \cite{suzuki2017}.

Parallel to these spatial generalizations, a major conceptual shift occurred in how contraction conditions are framed. Standard contractive mappings evaluate distance transitions strictly across pairs of points. To capture higher-order geometric relations, several authors proposed multi-point analogues. Recently, Petrov \cite{petrov2023a} introduced an innovative geometric framework based on the perimeters of triangles. A mapping $f:X\to X$ is classified as a contracting perimeters of triangles mapping (CPTM) if it scales down the total pairwise distance, the perimeter of any triangle formed by three mutually distinct vertices:
\begin{equation}
P(f(x),f(y),f(z))\leq qP(x,y,z),\quad q\in [0,1),
\label{eq:mcpT_intro}
\end{equation}
where $P(x,y,z)\coloneqq d(x,y) + d(y,z) + d(z,x)$. Petrov demonstrated that while the Banach contraction principle follows as a natural corollary when $s = 1$, the multi-point perimetric condition is fundamentally independent. Crucially, an CPTM can accommodate non-unique fixed points (up to two) and does not strictly require the mapping to be pairwise contractive anywhere. However, to guarantee fixed-point existence, one must impose a structural constraint: the exclusion of periodic points of prime period two (2-cycles), which would otherwise allow the system to alternate endlessly without collapsing to a steady state.

Following this development, Cvetković \cite{cvetkovic2026} investigated the behavior of higher-order iterations of CPTM within classical metric spaces. A central finding was that while an CPTM itself might not contract individual trajectories, its sufficiently high iterates $f^{n}$ inevitably behave as graphic (or orbital) contractions—a class of operators extensively formalized by Rus \cite{rus1993,rus2001,rus2003,rus2008} and Petruşel and Petruşel \cite{petrusel2023}. Graphic contractions satisfy:
\begin{equation}
d(f(x),f^{2}(x))\leq \alpha d(x,f(x))\quad \forall x\in X,
\label{eq:graphic_intro}
\end{equation}
ensuring that successive approximations along an operational path form a Cauchy sequence. Cvetković proved that for any CPTM lacking a 2-cycle, all iterates except finitely many are graphic contractions, establishing a profound bridge between multi-point geometric constraints and path-wise orbital stability within the theory of weakly Picard operators \cite{rus2003,petrusel2019}.

The convergence and stability of graphic contractions in generalized environments are governed by precise parametric guardrails. In a $b$-metric space, as shown by Petruşel and Petruşel \cite{petrusel2023}, a graphic contraction with parameter $q$ yields a stable, convergent Picard iteration if it obeys the tracking constraint $sq< 1$. When transitioning this theory to the perimetric framework within $b$-metric spaces, the interplay between the geometric compounding of the spatial parameter $s\geq 1$ and the multi-point contractive parameter $q\in [0,1)$ poses severe structural complications.

Motivated by these challenges, the purpose of this paper is to generalize, unify and extend the study of perimetric contractions and their iterates to the framework of complete $b$-metric spaces. We systematically investigate the architectural relationship between CPTM and graphic contractions under the presence or absence of periodic orbits, establishing precise threshold parameters. The contributions of this work are organized as follows:
\begin{enumerate}
\item We establish that under the minimal assumption of excluding prime period-two orbits, the iterates of an CPTM in a $b$-metric space form a graphic contraction for any iterate index $n$ satisfying the critical threshold $sq^n < 1$.
\item We prove a unified fixed-point existence and cardinality theorem for CPTM in complete $b$-metric spaces, demonstrating that the fixed-point set satisfies $1 \leq |\operatorname{Fix}(f)| \leq 2$.
\item We address the alternative scenario where the operator does possess a periodic orbit of prime period two, providing explicit structural criteria (under the parameter guardrail $sq^2 < 1$) where the even iterates $f^{2n}$ collapse into graphic contractions, leading to exactly two periodic points.
\item We present concrete analytical examples and counterexamples including examining the shift map on discrete metric configurations to prove that the class of CPTM is strictly larger than graphic contractions, and that our parameter boundaries are optimal.
\end{enumerate}
By synthesizing multi-point perimetric geometry with the operator-theoretic mechanics of weakly Picard operators, this manuscript unifies and completes the structural landscape of orbital stability in generalized metric topologies.

\section{Preliminaries}
In this section, we present some basic topological and operator-theoretic concepts needed in the sequel. The set $\mathbb{N} \coloneqq \{0, 1, 2, \ldots \}$ denotes the set of all natural numbers, and set $\mathbb{N}^* \coloneqq \mathbb{N} \setminus \{0\}$. The set $\mathbb{R}_+$ denotes the set of non-negative real numbers.

\subsection{Topological Structure and Challenges in $b$-Metric Spaces}
We begin by recalling the definition of a $b$-metric space, which generalizes the classical concept of a metric space by relaxing the triangle inequality using a topological multiplier $s \geq 1$.

\begin{definition}
\label{def:bmetric}
(Bakhtin \cite{bakhtin1989}, Czerwik \cite{czerwik1993}) Let $X$ be a nonempty set and let $s \geq 1$  a given real number. A functional $d: X \times X \to \mathbb{R}_+$ is said to be a $b$-metric with constant $s$ if for all $x, y, z \in X$, the following conditions are satisfied:
\begin{align}
d(x,y) &= 0 \iff x = y; \label{eq:bmetric_zero}\\
d(x,y) &= d(y,x); \label{eq:bmetric_sym}\\
d(x,z) &\leq s[d(x,y) + d(y,z)] \qquad (b\text{-triangle inequality}). \label{eq:bmetric_tri_def}
\end{align}
The pair $(X,d)$ is called a $b$-metric space with coefficient $s$.
\end{definition}

Though the notion of convergence, Cauchy sequences, and completeness are defined analogously to the standard metric spaces but the topology of $b$-metric spaces presents some significant analytical subtleties. The most important property is the continuity of distance function which induces a topology on the underlying space. Indeed, the $b$-metric distance function $d(\cdot,\cdot)$ is not generally continuous in both variables simultaneously when $s > 1$. To handle limit transitions in dealing with perimetric inequalities without imposing the strong assumption of continuity, we require the following fundamental important lemma.

\begin{lemma}
\label{lem:tracking}
(Aghajani et al. \cite{aghajani2014}, Woldegiorgis et al. \cite{woldegiorgis2025}) Let $(X,d)$ be a $b$-metric space with coefficient $s \geq 1$. If $\{x_n\}$ and $\{y_n\}$ are sequences in $X$ such that $\lim_{n \to \infty} x_n = x$ and $\lim_{n \to \infty} y_n = y$, then:
\begin{equation}
\frac{1}{s^2} d(x,y)\leq \liminf_{n\to \infty}d(x_n,y_n)\leq \limsup_{n\to \infty}d(x_n,y_n)\leq s^2 d(x,y).
\label{eq:tracking}
\end{equation}
In particular, if $y_n = y$ is a constant sequence, the above inequality reduces to:
\begin{equation}
\frac{1}{s} d(x,y)\leq \liminf_{n\to \infty}d(x_n,y)\leq \limsup_{n\to \infty}d(x_n,y)\leq s d(x,y).
\label{eq:tracking_const}
\end{equation}
\end{lemma}

\begin{remark}
Lemma \ref{lem:tracking} demonstrates why traditional metric fixed-point proofs fail in $b$-metric spaces when passing directly to limits across expressions involving mutually distinct points. If $x_n \to x$, and $y_n \to y,$ we cannot simply write $\lim d(x_n, y_n) = d(x,y)$ without the careful bounding provided above.
\end{remark}

\subsection{Weakly Picard Operators and Graphic Contractions}
To connect multi-point perimetric conditions to orbital stability, we present a suitable framework to study abstract fixed-point theory. Let us start with the theory of weakly Picard operators established by (Rus \cite{rus1993}).

\begin{definition}
\label{def:wpo}
(Rus \cite{rus1993}) Let $(X,d)$ be a $b$-metric space. A mapping $f:X\to X$ is said to be a \emph{weakly Picard operator} (WPO) if the sequence of successive approximations $\{f^{n}(x)\}$ converges for all $x\in X$ and its limit is a fixed point of $f$.
\end{definition}

A highly effective pathway for establishing that an operator belongs to the class of WPOs is proving that it, or a sufficiently high iterate, contracts along its operational paths. This behavior is well captured by the notion of graphic contractions.

\begin{definition}
\label{def:graphic}
(Petruşel \& Petruşel \cite{petrusel2023}) Let $(X,d)$ be a $b$-metric space and $f:X\to X$ a mapping. Let $G_{f}\coloneqq \{(x,f(x)):x\in X\}$ denote the graph of $f$. The mapping $f$ is called a \emph{graphic contraction} (or orbital contraction) if there exists a constant $q\in [0,1)$ such that:
\begin{equation}
d(f^{2}(x),f(x))\leq qd(f(x),x)\quad \forall x\in X.
\label{eq:graphic_def}
\end{equation}
\end{definition}

The fundamental convergence and uniqueness structure of graphic contractions in complete generalized spaces is governed by the following benchmark result.

\begin{theorem}
\label{thm:graphic_basic}
(Petruşel \& Petruşel \cite{petrusel2023}) Let $(X,d)$ be a complete $b$-metric space with coefficient $s\geq 1$ and $f:X\to X$ a graphic contraction with parameter $q\in [0,1)$ satisfying $sq< 1$. If $f$ possesses a closed graph $G_f$, then:
\begin{enumerate}
\item $\operatorname{Fix}(f)\neq \emptyset$;
\item For any $x_0\in X$, the sequence of iterates $x_n = f^n (x_0)$ converges to a fixed point $x^*\in \operatorname{Fix}(f)$.
\end{enumerate}
\end{theorem}

\begin{remark}
The standard Cauchy convergence estimate for graphic contractions in a $b$-metric space is given by
\begin{equation}
d(x_{n},x_{n + p})\leq \frac{s(sq)^{n}}{1 - sq} d(x_{0},x_{1}).
\label{eq:cauchy_estimate}
\end{equation}
This expression is stable if and only if the underlying parameters satisfy $sq< 1$. If $sq\geq 1$, the sequence may diverge or fail to be Cauchy due to the geometric compounding of the multiplier $s$.
\end{remark}

\subsection{Contracting Perimeters of Triangles Mappings (CPTM)}
We now formulate a single, unified definition for mappings that contract the perimeters of triangles in the framework of generalized metric spaces. This parameters-invariant specification naturally embeds standard metric spaces when setting $s = 1$.

For any three points $x,y,z\in X$, the perimeter of the triangle formed by these vertices is defined as the sum of their total pairwise distances:
\begin{equation}
P(x,y,z)\coloneqq d(x,y) + d(y,z) + d(z,x).
\label{eq:perimeter_def}
\end{equation}

\begin{definition}
\label{def:mcpT}
Let $(X,d)$ be a $b$-metric space with coefficient $s\geq 1$. A mapping $f:X\to X$ is said to be a \emph{ contracting perimeters of triangles mapping} (CPTM) if there exists a constant $q\in [0,1)$ such that for all $x,y,z\in X$ satisfying the distinctness condition $|\{x,y,z\}| = 3$, the following contractive inequality holds:
\begin{equation}
P(f(x),f(y),f(z))\leq qP(x,y,z). \label{eq:mcpT}
\end{equation}
\end{definition}

The strict demand in Definition \ref{def:mcpT} that the domain vertices $x,y,z$ must be pairwise distinct is a structural requirement of a perimetric theory. If two or more points coincide, $P(x,x,z) = 2d(x,z)$, and the expression collapses into a standard pairwise metric contraction, which changes the geometric scope of the operator.

To manipulate perimetric bounds under the $b$-triangle inequality, we will frequently utilize the following algebraic properties relating perimeters to basic pairwise distances.

\begin{lemma}
\label{lem:perimeter_relations}
Let $(X,d)$ be a $b$-metric space with coefficient $s\geq 1$. For any mutually distinct points $x,y,z\in X$, the following perimeter relations hold:
\begin{align}
d(x,y) &< P(x,y,z); \label{eq:perim_lb}\\
P(x,y,z) &\leq (s + 1)[d(x,y) + d(y,z)]. \label{eq:perim_ub}
\end{align}
\end{lemma}

\begin{proof}
The first property follows directly from the non-negativity of the $b$-metric and the strict distinctness of the points, ensuring $d(y,z) + d(z,x) > 0$. For the second property, expanding $P(x,y,z)$ and applying the $b$-triangle inequality to the third term gives
\begin{align*}
P(x,y,z) &= d(x,y) + d(y,z) + d(z,x)\\
&\leq d(x,y) + d(y,z) + s[d(z,y) + d(y,x)].
\end{align*}
Grouping terms based on their respective symmetric distances, we have
\[
P(x,y,z)\leq (1 + s)d(x,y) + (1 + s)d(y,z) = (s + 1)[d(x,y) + d(y,z)],
\]
which completes the proof.
\end{proof}

\begin{remark}
Note that there exist graphic contractions that are not contracting perimeters of triangles mappings and do not have any iterate that is a contracting perimeters of triangles mapping.
\end{remark}

\begin{example}
\label{ex:identity}
Let $(X,d)$ be a $b$-metric space with constant $s\geq 1$, and $f:X\to X$ an identity mapping, that is, $f(x) = x$ for all $x\in X$. Then, for any $n\in \mathbb{N}$, we have $f^{n}(x) = x$ for all $x\in X$, and $f$ is a graphic contraction since, for any $q\in [0,1)$ and all $x\in X$,
\[
d(f(x),f^{2}(x)) = d(x,x) = 0\leq qd(x,f(x)) = qd(x,x) = 0.
\]
Thus the graphic contraction inequality holds trivially.

Note that, $f$ is not a  contracting perimeters of triangles mapping (CPTM). Indeed, for any three mutually distinct points $x,y,z\in X$, we have
\[
d(f(x),f(y)) + d(f(y),f(z)) + d(f(z),f(x)) = d(x,y) + d(y,z) + d(z,x).
\]
For $f$ to be an CPTM, there would need to exist a constant $q\in [0,1)$ such that
\[
d(x,y) + d(y,z) + d(z,x)\leq q[d(x,y) + d(y,z) + d(z,x)]
\]
for all mutually distinct $x,y,z\in X$. Since the sum $d(x,y) + d(y,z) + d(z,x) > 0$ for mutually distinct points (as $d$ is a $b$-metric and hence separates points) which implies that $1\leq q$, which contradicts the assumption that $q< 1$. Therefore, \eqref{eq:mcpT} does not hold for any $q\in [0,1)$, and hence $f$ is not a contracting perimeters of triangles mapping. The same argument applies to every iterate $f^{n}$ (which is again the identity mapping), so no iterate of $f$ is an CPTM either.
\end{example}

\begin{remark}
Moreover, the converse does not hold in general. Indeed, there are contracting perimeters of triangles mapping that are not graphic contractions.
\end{remark}

\begin{example}
\label{ex:not_graphic}
Let $X = \{a,b,c\}$ be a set with three distinct elements. Define a function $d:X\times X\to [0,\infty)$ by
\[
d(a,c) = d(c,a) = 3,\quad d(a,b) = d(b,a) = 1,\quad d(b,c) = d(c,b) = 1,
\]
and $d(x,x)=0$ for all $x\in X$. Note that $(X,d)$ is a $b$-metric space with constant $s = 2$. The only nontrivial triangle-type inequality is given by
\[
d(a,c) = 3\leq 2\big(d(a,b) + d(b,c)\big) = 4,
\]
whereas all other triples are degenerate or satisfy the classical triangle inequality. Thus the $b$-metric axioms hold.

Define a mapping $f:X\to X$ by
\[
f(a) = b,\qquad f(b) = c,\qquad f(c) = c.
\]
Clearly, $c$ is a fixed point.
\begin{enumerate}
\item $f$ is a contracting perimeters of triangles mapping (CPTM). Since $X$ has exactly three distinct points, the only mutually distinct triple is $(a,b,c)$. We compute the perimeter of the image triple:
\[
d(f(a),f(b)) + d(f(b),f(c)) + d(f(c),f(a)) = d(b,c) + d(c,c) + d(c,b) = 1 + 0 + 1 = 2.
\]
The perimeter of the original triple is
\[
d(a,b) + d(b,c) + d(c,a) = 5.
\]
We need a constant $q\in [0,1)$ such that $2\leq q\cdot 5$. Indeed $q = \frac{2}{5} = 0.4$ serves the purpose. Hence $f$ satisfies the perimeter contraction condition with $q = 2/5$.

\item Note that, $f$ is not a graphic contraction. Recall that a graphic contraction requires the existence of $\alpha \in [0,1)$ such that
\[
d(f(x),f^{2}(x))\leq \alpha d(x,f(x))\quad \text{for all } x\in X.
\]
Take $x = a$. Then $f(a) = b$, $f^{2}(a) = f(b) = c$. Thus
\[
d(f(a),f^{2}(a)) = d(b,c) = 1,\qquad d(a,f(a)) = d(a,b) = 1.
\]
Suppose that $f$ is a graphic contraction. Then, we have
\[
1\leq \alpha \cdot 1 \Longrightarrow \alpha \geq 1,
\]
which contradicts $\alpha < 1$. Therefore, no such $\alpha$ exists, and $f$ is not a graphic contraction.
\end{enumerate}
Hence, even in the $b$-metric setting, the class of contracting perimeters of triangles mapping is strictly larger than the class of graphic contractions: there exist CPTM that are not graphic contractions.
\end{example}

We prove that such mappings cannot have periodic orbits of length greater than two, a property essential for establishing fixed point theorems in complete $b$-metric spaces.

\begin{lemma}
\label{lem:no_period_gt2}
Let $(X,d)$ be a $b$-metric space with constant $s\geq 1$. If $f:X\to X$ is a contracting perimeters of triangles mapping, that is, a mapping satisfying \eqref{eq:mcpT}, then it has no periodic points of a prime period greater than two.
\end{lemma}

\begin{proof}
Assume on contrary that $x\in X$ is a periodic point of the mapping $f$ with a prime period $k$ for some $k\geq 3$.

By the definition of a prime period, the elements within the finite orbit are distinct up to the $k$-th iterate. That is, $i\neq j$ implies $f^{i}(x)\neq f^{j}(x)$ for all $i,j\in \{0,1,\ldots,k - 1\}$. Consequently, for any integer $n$, the triple of consecutive iterates
\[
(f^{n}(x),f^{n + 1}(x),f^{n + 2}(x))
\]
consists of mutually distinct points.

In a $b$-metric space, the identity of indiscernibles holds exactly as it does in standard metric spaces:
\[
d(u,v) = 0\iff u = v.
\]
Since the points $x,f(x)$, and $f^{2}(x)$ are mutually distinct, their pairwise distances must be strictly positive:
\begin{equation}
d(x,f(x)) > 0,\quad d(f(x),f^{2}(x)) > 0,\quad \text{and}\quad d(f^{2}(x),x) > 0.
\label{eq:dist_pos}
\end{equation}
Thus, the perimeter of the triangle formed by these three points is strictly positive:
\begin{equation}
P(x,f(x),f^{2}(x))\coloneqq d(x,f(x)) + d(f(x),f^{2}(x)) + d(f^{2}(x),x) > 0.
\label{eq:perim_pos}
\end{equation}

As every consecutive triple along the orbit consists of mutually distinct points, the perimeter contraction condition \eqref{eq:mcpT} is applicable  at each step. Applying this condition successively $k$ times along the periodic cycle yields:
\begin{align}
P(f^{k}(x),f^{k + 1}(x),f^{k + 2}(x)) &\leq q P(f^{k - 1}(x),f^{k}(x),f^{k + 1}(x)) \notag\\
&\leq q^{2}P(f^{k - 2}(x),f^{k - 1}(x),f^{k}(x)) \notag\\
&\qquad \vdots \notag\\
&\leq q^{k}P(x,f(x),f^{2}(x)).
\label{eq:period_chain}
\end{align}

Since $x$ is a periodic point of period $k$, we have
\[
f^{k}(x) = x,\qquad f^{k + 1}(x) = f(x),\qquad f^{k + 2}(x) = f^{2}(x).
\]
Substituting these periodic identities into the left-hand side of the chain \eqref{eq:period_chain}, we obtain:
\begin{equation}
P(x,f(x),f^{2}(x))\leq q^{k}P(x,f(x),f^{2}(x)).
\label{eq:period_ineq}
\end{equation}
From \eqref{eq:perim_pos}, we have $P(x,f(x),f^{2}(x)) > 0$. Dividing both sides of \eqref{eq:period_ineq} by this positive quantity yields:
\begin{equation}
1\leq q^{k}.
\label{eq:qk}
\end{equation}
However, the definition of a perimeter contraction requires the contractive constant to satisfy $q\in [0,1)$. For any $k\geq 3$, it necessarily holds that $q^{k}<1$, which contradicts \eqref{eq:qk}.

Thus, our assumption leads to a contradiction. Hence, a contracting perimeters of triangles mapping on a $b$-metric space cannot possess any periodic points of a prime period greater than two.
\end{proof}

\begin{remark}
It is worth mentioning that a constant $s\geq 1$ of the $b$-metric does not appear in the proof given above, indicating that the lemma holds uniformly for all $b$-metric spaces, regardless of the coefficient. This is a notable feature of perimeter contractions which distinguish them  from standard Banach-type contractions whose behavior often depends on $s$.
\end{remark}

\section{Main Results}

In this section, we first show that global continuity is an intrinsic geometric property of any CPTM in a $b$-metric space, requiring neither the completeness of the space nor any restrictions on periodic orbits. Separating this regularizing feature as an important proposition facilitates the presentation of our main result.

\begin{proposition}
\label{prop:continuity}
Let $(X,d)$ be a $b$-metric space with parameter $s\geq 1$. If $f:X\to X$ is a contracting perimeters of triangles mapping (CPTM), then $f$ is continuous on $X$.
\end{proposition}

\begin{proof}
Let $x\in X$ be an arbitrary point, and let $(y_{m})_{m\in \mathbb{N}}$ be a sequence in $X$ converging to $x$, so that $\lim_{m\to \infty}d(y_{m},x) = 0$. To prove continuity, it suffices to show that $\lim_{m\to \infty}f(y_{m}) = f(x)$.

Assume for contrary that $f(y_{m})\not\to f(x)$. By definition of non-convergence, there exist an $\epsilon >0$ and a subsequence $(y_{m_k})_{k\in \mathbb{N}}$ such that
\[
d(f(y_{m_k}),f(x))\geq \epsilon \quad \forall k\in \mathbb{N}.
\]
We now divide the proof into following four steps.\newline
\textbf{Step 1}:  We first establish that we can extract a subsequence of $(y_{m_k})$ consisting of pairwise distinct points, all different from $x$.

If $y_{m_k} = x$ for infinitely many indices, then for those indices we would have $d(f(y_{m_k}),f(x)) = d(f(x),f(x)) = 0$, directly contradicting the assumption that $d(f(y_{m_k}),f(x))\geq \epsilon$. Thus, $x$ can appear at most finitely many times in the sequence $(y_{m_k})$. Discarding these finite terms allows us to assume without loss of generality that $y_{m_k}\neq x$ for all $k\in \mathbb{N}$. Next, we claim that any unique value $z\in X\backslash \{x\}$ can appear at most finitely many times in $(y_{m_k})$. Suppose some $z\neq x$ appears infinitely many times. Then there exists a subsequence identically equal to $z$. Because any subsequence of a convergent sequence must share the same limit, this constant subsequence must converge to $x$. By the identity of indiscernibles, this implies $d(z,x) = 0$, yielding $z = x$, which is a contradiction.

Since each unique value occurs only finitely many times across an infinite sequence, the set of values $\{y_{m_k}:k\in \mathbb{N}\}$ must be infinite. Therefore, we can inductively extract a subsequence which we reindex and denote as $(y_{m_i})_{i\in \mathbb{N}}$ such that $y_{m_i}\neq y_{m_j}$ for all $i\neq j$. Thus, the set $\{x\} \cup \{y_{m_i}:i\in \mathbb{N}\}$ consists of mutually distinct points.

\textbf{Step 2}: This step deals with the finite image set case (Pigeonhole Principle).\newline Let $F = \{f(y_{m_i}):i\in \mathbb{N}\}$ be the set of image values under $f$. Suppose on contrary that $F$ is a finite set.

By the Pigeonhole Principle, since the index set $\mathbb{N}$ is infinite and $F$ is finite, there must exist a specific point $c\in X$ and an infinite subsequence of $(y_{m_i})$ (which we still denote as $(y_{m_i})$ for simplicity) such that $f(y_{m_i}) = c$ for all $i\in \mathbb{N}$. As $d(f(y_{m_i}),f(x))\geq \epsilon >0$, $c\neq f(x)$.

Now, select any two distinct indices $i\neq j$. The domain elements $x,y_{m_i},y_{m_j}$ are mutually distinct by construction. Applying the CPTM contractive condition to this triple, we have
\[
P(f(x),f(y_{m_i}),f(y_{m_j}))\leq qP(x,y_{m_i},y_{m_j}).
\]
Substituting $f(y_{m_i}) = f(y_{m_j}) = c$ into the left-hand side, we get
\[
P(f(x),c,c) = d(f(x),c) + d(c,c) + d(c,f(x)) = 2d(f(x),c).
\]
Using the $b$-triangle inequality to expand the domain perimeter on the right-hand side, and noting that $d(y_{m_i},y_{m_j})\leq s[d(y_{m_i},x) + d(x,y_{m_j})]$, the inequality becomes
\[
2d(f(x),c)\leq q(s + 1)\left[d(x,y_{m_i}) + d(x,y_{m_j})\right].
\]
As this holds for all distinct $i$ and $j$, on taking limit as $i,j\to \infty$ and applying the Squeeze Theorem (since $y_{m_i},y_{m_j}\to x$), we obtain that
\[
2d(f(x),c)\leq q(s + 1)[0 + 0] = 0\Rightarrow d(f(x),c) = 0\Rightarrow c = f(x),
\]
a contradiction.

\textbf{Step 3}: This step deals with formal inductive construction for the infinite image case.\newline As $F$ is infinite, we can construct a further subsequence whose images are completely distinct from each other and from $f(x)$. We define this subsequence $(y_{m_{i_k}})_{k \in \mathbb{N}}$ inductively:

\textit{Base Step $(k = 1)$:} Choose $i_{1} = 1$. Since $F$ is infinite, $F\backslash \{f(x)\}$ remains infinite; hence, we can ensure $f(y_{m_{i_1}})\neq f(x)$.\newline \textit{Inductive Step:} Assume we have successfully chosen indices $i_1< i_2< \dots < i_k$ such that the set $\{f(y_{m_{i_1}}),\ldots,f(y_{m_{i_k}})\}$ consists of $k$ distinct values, none of which equal $f(x)$. Consider the tail set $F_{k} = \{f(y_{m_{i}}):i > i_{k}\}$. Removing a finite number of elements from the infinite set $F$ leaves $F_{k}$ infinite. Therefore, $F_{k}$ cannot be contained within the finite set $V_{k} = \{f(x),f(y_{m_{i_1}}),\ldots,f(y_{m_{i_k}})\}$. Thus, there must exist an index $i_{k + 1} > i_{k}$ such that $f(y_{m_{i_{k + 1}}})\notin V_{k}$.

By mathematical induction, we obtain a subsequence where both the domain elements are pairwise distinct and their images are pairwise distinct and omit $f(x)$. To simplify notation, let us re-index this and obtain a subsequence as $(w_{k})_{k \in \mathbb{N}}$.

\textbf{Step 4}: This step is about analytical Squeezing and limit Passing.\newline For any two distinct indices $k \neq l$, the domain triple $(x, w_k, w_l)$ and the image triple $(f(x), f(w_k), f(w_l))$ both consist of mutually distinct points. Applying the perimetric contraction condition yields
\[
P(f(x),f(w_k),f(w_l))\leq qP(x,w_k,w_l).
\]
Expanding the perimeters gives
\[
d(f(x),f(w_k)) + d(f(w_k),f(w_l)) + d(f(w_l),f(x))\leq q[d(x,w_k) + d(w_k,w_l) + d(w_l,x)].
\]
Bounding the right-hand side via $d(w_k, w_l) \leq s[d(w_k, x) + d(x, w_l)]$ leads to
\[
d(f(x),f(w_k)) + d(f(w_k),f(w_l)) + d(f(w_l),f(x))\leq q(s + 1)[d(x,w_k) + d(x,w_l)].
\]
Since all metrics on the left-hand side are strictly non-negative, we can drop the second and third terms to isolate the first term:
\begin{equation}
d(f(x),f(w_k))\leq q(s + 1)[d(x,w_k) + d(x,w_l)].
\label{eq:cont_step4}
\end{equation}
Inequality~\eqref{eq:cont_step4} holds strictly for a fixed index $k$ and any choice of $l \in \mathbb{N} \setminus \{k\}$. To evaluate this behavior as $l \to \infty$, we must account for the fact that the $b$-metric $d(\cdot, \cdot)$ is generally not continuous in both variables simultaneously. To rigorously bypass this structural limitation, we take the limit supremum as $l \to \infty$ on both sides of~\eqref{eq:cont_step4}. 

Since the left-hand side is completely independent of $l$, it behaves as a constant bound relative to the running index. Utilizing the fact that $\lim_{l \to \infty} d(x, w_l) = 0$, Lemma~2.2 justifies passing the limit into the right-hand side, yielding
\[
d(f(x), f(w_k)) \le \limsup_{l \to \infty} q(s + 1) \left[ d(x, w_k) + d(x, w_l) \right] = q(s + 1) d(x, w_k).
\]
Finally, taking the limit as $k \to \infty$ on both sides of this reduced inequality, and noting that $\lim_{k \to \infty} d(x, w_k) = 0$, the standard Squeeze Theorem gives 
\[
\lim_{k \to \infty} d(f(x), f(w_k)) = 0.
\]
This contradicts our  assumption that $d(f(w_k), f(x)) \ge \epsilon > 0$ for all $k \in \mathbb{N}$. Thus, our assumption is wrong, and hence $f$ is continuous at $x$. Since $x \in X$ was arbitrarily selected, $f$ is continuous on $X$.
\end{proof}

\begin{theorem}
\label{thm:main_nonperiodic}
Let $(X,d)$ be a complete $b$-metric space with constant $s\geq 1$ such that $|X|\geq 3$. If $f:X\to X$ is an CPTM with contraction constant $q\in [0,1)$ that has no periodic points of prime period two. Then there exists $n_0\in \mathbb{N}$ such that for all $n\geq n_0$ we have $sq^n< 1$, and $f^n$ is a continuous graphic contraction on $(X,d)$ with coefficient
\begin{equation}
\lambda_{n} = \frac{sq^{n}}{1 - sq^{n}} < 1.
\label{eq:lambda_n}
\end{equation}
\end{theorem}

\begin{proof}
	We divide the proof of the theorem into three steps.\\
\textbf{Step 1}:  Since $q\in [0,1)$, $\lim_{n\to \infty}s q^{n} = 0$. Hence there is $n_0$ such that for all $n\geq n_0$, $sq^n< 1$; then $\lambda_{n}$ is well-defined and lies in $[0,1)$.

\textbf{Step 2}: In this step, we give graphic contraction estimate for $f^n$ ($n\geq n_0$). Fix $n\geq n_0$ and $x\in X$. Consider the triple $(x, f^n x, f^{2n} x)$.

If the points are not mutually distinct, then at least two coincide. We claim that in all cases this forces $d(f^n x, f^{2n} x)=0$. Indeed:
\begin{itemize}
\item If $x = f^n x$, then $x$ is periodic with period dividing $n$. Since no period 2 exists and by Lemma \ref{lem:no_period_gt2} the prime period cannot exceed 2, the period must be 1, so $x$ is a fixed point, and hence $f^n x = x$ and $f^{2n}x = x$, giving distance equal to 0.
\item If $x = f^{2n} x$, then $x$ is periodic with period dividing $2n$. Following arguments similar to those given above, the period is 1, so $f^n x = x$, again gives the  distance equal to 0.
\item If $f^n x = f^{2n} x$, then, it is trivial that distance is equal to 0.
\end{itemize}
Thus in all degenerate cases the graphic contraction inequality holds with left side 0.

Now assume $x, f^n x, f^{2n} x$ are mutually distinct. Then we can apply the perimeter contraction condition $n$ times to the consecutive triples, yielding:
\begin{equation}
P(f^n x, f^{2n} x, f^{3n} x) \leq q^n P(x, f^n x, f^{2n} x).
\label{eq:main_triple}
\end{equation}
Using Lemma \ref{lem:perimeter_relations} and the $b$-triangle inequality exactly as in the proof of Theorem 3.2 (lower bounding the left side by $\frac{s+1}{s} d(f^n x, f^{2n} x)$ and upper bounding the right side by $(s+1)[d(x,f^n x)+d(f^n x, f^{2n} x)]$), we obtain:
\[
d(f^n x, f^{2n} x) \leq sq^n [d(x,f^n x) + d(f^n x, f^{2n} x)].
\]
Rearranging and using $sq^n < 1$ yields:
\begin{equation}
d(f^n x, f^{2n} x) \leq \frac{sq^n}{1 - sq^n} d(x, f^n x).
\label{eq:main_graphic}
\end{equation}
Thus $f^n$ is a graphic contraction with coefficient $\lambda_n$.

\textbf{Step 3}: We now prove that  the graph of $f$ is closed and the existence of a fixed point of the mapping $f$. By Proposition \ref{prop:continuity}, $f$ is continuous on $X$. Since $b$-metric spaces are Hausdorff (the metric separates points) \cite{Kunzi2022}, the graph of a continuous map into a Hausdorff space is closed [Theorem 2.9.3 of \cite{malkowsky2019}]. Hence $G_f$ is closed.

Since $f^{n_0}$ is a graphic contraction with $s\lambda_{n_0}<1$, Theorem \ref{thm:graphic_basic} applies and guarantees that the Picard iterates converge to a fixed point of $f^{n_0}$. Let $x^*\in \operatorname{Fix}(f^{n_0})$. Then $x^*$ is periodic for $f$; its prime period divides $n_0$. By Lemma \ref{lem:no_period_gt2} and the exclusion of period 2, the period is 1, so $f(x^*)=x^*$. Thus $x^*\in \operatorname{Fix}(f)$, completing the proof.
\end{proof}

\begin{example}
\label{ex:shift_map}
Let $X = \{x_{n} \mid n \in \mathbb{N}_{0}\}$ be a set of distinct points. Define a symmetric function $d: X \times X \to \mathbb{R}_{+}$ by $d(x_{i}, x_{j}) = 0$ for all $i$, and for $i < j$:
\[
d(x_i, x_j) =
\begin{cases}
10, & (i,j) = (0,1), (1,2), (1,3),\\
14, & (i,j) = (0,2),\\
15, & (i,j) = (0,3), (2,3),\\
3^j, & j \geq 4.
\end{cases}
\]
For $j<4$ the distances are:
\[
d(x_{0},x_{1}) = 10,\quad d(x_{1},x_{2}) = 10,\quad d(x_{1},x_{3}) = 10,\quad d(x_{0},x_{2}) = 14,\quad d(x_{0},x_{3}) = 15,\quad d(x_{2},x_{3}) = 15.
\]
We now make and prove the following claims:\newline
\noindent\textbf{Claim 1:} $d$ is a metric.

\noindent The function $d$ satisfies the triangle inequality.

\textit{Proof.} Take three mutually distinct points $x_{i},x_{j},x_{k}$ with $i< j< k$.

Case 1: If $k\geq 4$. Then $d(x_{i},x_{k}) = d(x_{j},x_{k}) = 3^{k}$. We check the three inequalities:
\begin{enumerate}
\item $d(x_{i},x_{j})\leq d(x_{i},x_{k}) + d(x_{j},x_{k}) = 2\cdot 3^{k}$. If $j\geq 4$, then $d(x_{i},x_{j}) = 3^{j}\leq 3^{k-1}<2\cdot 3^{k}$; if $j<4$, then $d(x_{i},x_{j})\leq 15\leq 2\cdot 3^{4}\leq 2\cdot 3^{k}$.
\item $d(x_{i},x_{k})\leq d(x_{i},x_{j}) + d(x_{j},x_{k})$ becomes $3^{k}\leq d(x_{i},x_{j}) + 3^{k}$, which is true.
\item $d(x_{j},x_{k})\leq d(x_{i},x_{j}) + d(x_{i},x_{k})$ which is true.
\end{enumerate}

Case 2: $k<4$. All points belong to $\{x_{0},x_{1},x_{2},x_{3}\}$. Direct verification gives the triangle inequality for all triples. Thus the triangle inequality holds in all cases, so $(X,d)$ is a metric space.

Now fix $p\geq 1$ and define a new distance given by
\[
d_{b}(x,y)\coloneqq \left(d(x,y)\right)^{p}.
\]
By the elementary inequality $(a + b)^{p}\leq 2^{p - 1}(a^{p} + b^{p})$, $(X,d_{b})$ is a $b$-metric space with parameter $s = 2^{p - 1}$ (see \cite{negash2025}). When $p > 1$, this is a genuine $b$-metric (the usual triangle inequality may fail).

\noindent\textbf{Claim 2:} The $b$-metric space $(X,d_{b})$ is complete.

\textit{Proof.} For any distinct $x_{i},x_{j}$, we have $d(x_{i},x_{j})\geq 10$, hence $d_{b}(x_{i},x_{j})\geq 10^{p} > 0$. Let $(y_{n})$ be a Cauchy sequence. Choose $0< \epsilon < 10^{p}$. Then there exists $N$ such that $d_{b}(y_{m},y_{n})< \epsilon$ for all $m,n\geq N$. Since any two distinct points are separated by at least $10^{p}$, this forces $y_{m} = y_{n}$ for all $m,n\geq N$. Thus the sequence is eventually constant and converges in $X$. Hence $(X,d_{b})$ is complete. $\square$

Now define the shift map $f:X\to X$ by
\[
f(x_0) = x_0,\qquad f(x_n) = x_{n - 1}\quad (n\geq 1).
\]
The unique fixed point is $x_{0}$, and $f$ has no periodic points of prime period 2 (or any period $>1$), since every positive index strictly decreases under $f$.

\noindent\textbf{Claim 3:} The mapping $f$ is a perimeter contracting mapping.

There exists a constant $q<1$ such that for all mutually distinct $x_{i},x_{j},x_{k}\in X$
\[
P(fx_{i},fx_{j},fx_{k})\leq qP(x_{i},x_{j},x_{k}),
\]
where $P(u,v,w) = d_{b}(u,v) + d_{b}(v,w) + d_{b}(w,u)$.

To prove this, we proceed as follows:\newline Order the indices as $i< j< k$. Define the perimeter ratio
\[
R(i,j,k)\coloneqq \frac{P(fx_{i},fx_{j},fx_{k})}{P(x_{i},x_{j},x_{k})}.
\]
We partition the triples into two cases.

Case: $k\geq 5$. Then $d(x_{i},x_{k}) = d(x_{j},x_{k}) = 3^{k}$, so $P(x_{i},x_{j},x_{k}) = d_{b}(x_{i},x_{j}) + 2\cdot 3^{k p}$. Under $f$, the largest index becomes $k - 1\geq 4$, and
\[
P(fx_{i},fx_{j},fx_{k}) = d_{b}(fx_{i},fx_{j}) + 2\cdot 3^{(k - 1)p}.
\]
Thus
\[
R(i,j,k) = \frac{d_{b}(fx_{i},fx_{j}) + 2\cdot 3^{(k - 1)p}}{d_{b}(x_{i},x_{j}) + 2\cdot 3^{kp}}.
\]
If $j\geq 5$, then $d_{b}(fx_{i},fx_{j}) = 3^{(j - 1)p} = 3^{- p}d_{b}(x_{i},x_{j})$, and $R = 3^{- p}$.

If $j\leq 4$, then $d_{b}(x_{i},x_{j})\geq 10^{p}$ and $d_{b}(fx_{i},fx_{j})\leq 15^{p}$. Set $u = 3^{(k - 1)p}\geq 3^{4p} = 81^{p}$. Then
\[
R(i,j,k)\leq \frac{15^{p} + 2u}{10^{p} + 2\cdot 3^{p}\cdot u} =: g(u).
\]
The function $g$ is strictly decreasing because its derivative is
\[
g^{\prime}(u) = \frac{2(10^{p}) - 2\cdot 3^{p}(15^{p})}{(10^{p} + 2\cdot 3^{p}u)^{2}} < 0.
\]
Thus the supremum over $u\geq 81^{p}$ is attained at $u = 81^{p}$:
\[
R(i,j,k)\leq \frac{15^{p} + 2\cdot 81^{p}}{10^{p} + 2\cdot 243^{p}} < 1.
\]

Core case: $k\leq 4$. There are only $\binom{5}{3} = 10$ triples from $\{x_{0},\ldots,x_{4}\}$. Each ratio $R(i,j,k)$ is strictly less than 1; for example, if we take $(0,1,4)$, then
\[
P(x_{0},x_{1},x_{4}) = 10^{p} + 2\cdot 81^{p},\quad P(fx_{0},fx_{1},fx_{4}) = P(x_{0},x_{0},x_{3}) = 2\cdot 15^{p},
\]
so $R = \frac{2\cdot 15^{p}}{10^{p} + 2\cdot 81^{p}} < 1$. Hence the maximum over these finitely many triples is $< 1$.

Taking
\[
q\coloneqq \max \left\{\underset {\text{core}}{\max}R,\frac{15^{p} + 2\cdot 81^{p}}{10^{p} + 2\cdot 243^{p}}\right\} < 1,
\]
we obtain the uniform perimeter contraction inequality.

We now examine whether $f^{n}$ is a graphic contraction for $n = 1,2,3$.

\noindent\textbf{Claim 4:} The mappings $f$ and $f^{2}$ are not graphic contractions.

\textit{Proof.} For $f$, take $x = x_{2}$. Then $f x = x_{1}$ and $f^{2}x = x_{0}$. The ratio is
\[
\frac{d_{b}(f^{2}x,f x)}{d_{b}(f x,x)} = \frac{d_{b}(x_{0},x_{1})}{d_{b}(x_{1},x_{2})} = \frac{10^{p}}{10^{p}} = 1.
\]
Hence, we may not find any $\lambda < 1$.

For $f^{2}$, take $x = x_{3}$. Then $f^{2}x = x_{1}$ and $f^{4}x = x_{0}$. The ratio is again 1:
\[
\frac{d_{b}(f^{4}x,f^{2}x)}{d_{b}(f^{2}x,x)} = \frac{d_{b}(x_{0},x_{1})}{d_{b}(x_{1},x_{3})} = \frac{10^{p}}{10^{p}} = 1.
\]
Thus $f^{2}$ is not graphic.

\noindent\textbf{Claim 5:} $f^{3}$ is graphic.

The third iterate $f^{3}$ is a graphic contraction with constant
\[
\lambda_{b} = \left(\frac{10}{81}\right)^{p}< 1.
\]

\textit{Proof.} We compute the ratio for all $x_{n}\in X$:

If $n\leq 3$, then $f^{3}x_{n} = x_{0}$ and $f^{6}x_{n} = x_{0}$, so the numerator is 0.

For $n = 4$: $f^{3}x_{4} = x_{1}$, $f^{6}x_{4} = x_{0}$, ratio is given as follows:
\[
\frac{d_{b}(x_{1},x_{0})}{d_{b}(x_{4},x_{1})} = \frac{10^{p}}{81^{p}} = \left(\frac{10}{81}\right)^{p}.
\]
For $n = 5$: $f^{3}x_{5} = x_{2}$, $f^{6}x_{5} = x_{0}$, ratio is given by
\[
\frac{d_{b}(x_{2},x_{0})}{d_{b}(x_{5},x_{2})} = \frac{14^{p}}{243^{p}} = \left(\frac{14}{243}\right)^{p}.
\]
For $n\geq 6$: $f^{3}x_{n} = x_{n - 3}$, $f^{6}x_{n} = x_{n - 6}$, we have
\[
\frac{d_{b}(x_{n - 3},x_{n - 6})}{d_{b}(x_{n},x_{n - 3})} = \frac{(3^{n - 3})^{p}}{(3^{n})^{p}} = 3^{-3p}.
\]
Thus the supremum of all ratios is
\[
\lambda_{b} = \max \left\{\left(\frac{10}{81}\right)^{p},\left(\frac{14}{243}\right)^{p},3^{-3p}\right\}.
\]
Since $\frac{10}{81} >\frac{14}{243} >\frac{1}{27}$, we have
\[
\lambda_{b} = \left(\frac{10}{81}\right)^{p}< 1.
\]
Hence $f^{3}$ is a graphic contraction.
\end{example}

The existence of a fixed point of a contracting perimeters of triangles mapping without periodic points of a prime period two on a complete $b$-metric space can be directly derived from Theorem \ref{thm:graphic_basic}.

\begin{corollary}
\label{cor:fixedpoint}
If $f:X\to X$ is a  contracting perimeters of triangles mapping without periodic points of a prime period two on a complete $b$-metric space $(X,d)$ with constant $s\geq 1$, then it possesses a fixed point in $X$.
\end{corollary}

\begin{proof}
By Theorem \ref{thm:main_nonperiodic}, there exists $n_{0}\in \mathbb{N}$ such that $f^{n_{0}}$ is an orbitally continuous graphic contraction. Since $(X,d)$ is complete, Theorem \ref{thm:graphic_basic} implies that $f^{n_{0}}$ possesses a fixed point $x^{*}\in X$; that is, $f^{n_{0}}(x^{*}) = x^{*}$. Hence $x^{*}$ is a periodic point of $f$. Let $k$ denote the prime period of $x^*$. Since $f^{n_0}(x^*) = x^*$, the division algorithm implies that $n_0 = qk + r$ for $0 \le r < k$. Thus
\[
x^* = f^{n_0}(x^*) = f^r(f^{qk}(x^*)) = f^r(x^*),
\]
which by the minimality of the prime period gives $r = 0$, proving that $k \mid n_0$. By Lemma \ref{lem:no_period_gt2}, $k\leq 2$. As period two is excluded, $k=1$, so $f(x^{*})=x^{*}$.
\end{proof}

\begin{lemma}
\label{lem:atmost2fixed}
Let $(X,d)$ be a $b$-metric space and let $T:X\to X$ satisfy the contractive condition \eqref{eq:mcpT}. Then $T$ possesses at most two fixed points.
\end{lemma}

\begin{proof}
Suppose that $T$ has three distinct fixed points $x,y,z\in X$. Since $Tx=x, Ty=y, Tz=z$, we obtain
\[
P(Tx,Ty,Tz) = P(x,y,z),
\]
where $P(x,y,z) = d(x,y) + d(y,z) + d(z,x)$. Because $x,y,z$ are mutually distinct and $d(u,v) = 0$ iff $u=v$, we have $P(x,y,z) > 0$. Applying the perimeter contraction condition to $x,y,z$ gives
\[
P(x,y,z) = P(Tx,Ty,Tz)\leq qP(x,y,z).
\]
Dividing by the positive quantity $P(x,y,z)$ yields $1\leq q$, contradicting $0\leq q<1$. Therefore, $T$ cannot have three distinct fixed points. Hence the number of fixed points of $T$ is at most two.
\end{proof}

\begin{remark}
\label{rem:optimal_fixed}
The above result is optimal, since there exist perimeter contractive mappings possessing exactly two distinct fixed points.
\end{remark}

\begin{remark}
\label{rem:fix_iter}
A careful examination of the proof of Corollary \ref{cor:fixedpoint} gives that
\[
\operatorname{Fix}(f^n) = \operatorname{Fix}(f)\qquad \text{for every } n\in \mathbb{N}.
\]
Indeed, if $x\in \operatorname{Fix}(f^n)$, then $x$ is a periodic point of $f$. Let $k$ denote its prime period. By Lemma \ref{lem:no_period_gt2}, the mapping $f$ admits no periodic points of prime period greater than two, while periodic points of prime period two are excluded by hypothesis. Consequently, $k=1$, and hence $f(x)=x$. Therefore, $\operatorname{Fix}(f^n)\subseteq \operatorname{Fix}(f)$. The reverse inclusion is immediate, since every fixed point of $f$ is clearly a fixed point of each iterate $f^n$. Thus, $\operatorname{Fix}(f^n) = \operatorname{Fix}(f)$ for all $n\in \mathbb{N}$. This property resembles the classical Banach contraction principle. However, unlike Banach contractions, mappings contracting perimeters of triangles may possess two distinct fixed points, and therefore uniqueness of the fixed point cannot be expected in general.
\end{remark}

As emphasized above, the absence of periodic points of prime period two is essential in Theorem \ref{thm:main_nonperiodic}. The following example shows that the perimeter contraction property alone does not prevent the existence of 2-cycles.

\begin{example}
\label{ex:2cycle}
(Adapted from \cite{cvetkovic2026} to $b$-metric spaces). Let
\[
X = \{a,b,c\}
\]
and equip $X$ with the discrete metric
\[
d(x,y) = 
\begin{cases}
0, & x=y,\\
1, & x\neq y.
\end{cases}
\]
Since every metric is a $b$-metric, $(X,d)$ is a complete $b$-metric space with parameter $s=1$. Define $f\colon X\to X$ by
\[
f(a) = b,\qquad f(b) = a,\qquad f(c) = a.
\]
Since $\{a,b,c\}$ is the only triple of mutually distinct points, we compute
\[
P(f(a),f(b),f(c)) = P(b,a,a) = d(b,a) + d(a,a) + d(a,b) = 2,
\]
whereas
\[
P(a,b,c) = d(a,b) + d(b,c) + d(c,a) = 3.
\]
Hence
\[
P(f(a),f(b),f(c)) = \frac{2}{3}P(a,b,c),
\]
showing that $f$ contracts perimeters of triangles with contractive constant $q = \frac{2}{3}$.

Furthermore,
\[
f^{2}(a) = a,\qquad f^{2}(b) = b,
\]
and therefore $a$ and $b$ form a periodic orbit of prime period two.

For every odd integer $n$,
\[
f^{n}(a) = b,\qquad f^{2n}(a) = a.
\]
Consequently,
\[
d(f^{n}(a),f^{2n}(a)) = d(b,a) = 1,
\]
while
\[
d(a,f^{n}(a)) = d(a,b) = 1.
\]
If $f^{n}$ were a graphical contraction, there would exist $\lambda \in [0,1)$ such that
\[
d(f^{n}(a),f^{2n}(a))\leq \lambda d(a,f^{n}(a)).
\]
Substituting the above values gives $1\leq \lambda$, which contradicts $\lambda<1$. Therefore, $f^{n}$ fails to be a graphical contraction for every odd integer $n$. This demonstrates that the conclusion of Theorem \ref{thm:main_nonperiodic} may fail in the presence of periodic points of prime period two.
\end{example}

\begin{remark}
The preceding example shows that the perimeter contraction property alone does not guarantee that sufficiently large iterates become graphical contractions. The sole obstruction is the existence of periodic points of prime period two. Consequently, the assumption excluding 2-cycles in Theorem \ref{thm:main_nonperiodic} is essential.
\end{remark}

Combining Corollary \ref{cor:fixedpoint} with the preceding discussion yields the following consequence.

\begin{corollary}
\label{cor:periodic_atmost2}
Let $f:X\to X$ be a  contracting perimeters of triangles mapping on a complete $b$-metric space $(X,d)$. Then $f$ possesses a periodic point whose prime period is at most two.
\end{corollary}

\begin{proof}
If $f$ has no periodic point of prime period two, then Corollary \ref{cor:fixedpoint} implies that $f$ has a fixed point, which is a periodic point of prime period one. Otherwise, $f$ possesses a periodic point of prime period two. Hence, in either case, $f$ admits a periodic point whose prime period is at most two.
\end{proof}

\begin{remark}
A fixed point and a periodic point of prime period two cannot coexist for a  contracting perimeters of triangles mapping. Consequently, such mappings exhibit a dichotomy: either their periodic points are fixed points, or they belong to 2-cycles, but both phenomena cannot occur simultaneously.
\end{remark}

\begin{theorem}
\label{thm:no_fixed_with_2cycle}
Let $(X,d)$ be a $b$-metric space and $f:X\to X$  a contracting perimeters of triangles mapping; that is, there exists a constant $q\in [0,1)$ such that
\[
d(fx,fy) + d(fy,fz) + d(fz,fx)\leq q[d(x,y) + d(y,z) + d(z,x)]
\]
for all mutually distinct points $x,y,z\in X$. Then $f$ cannot possess simultaneously a fixed point and a periodic point of prime period two.
\end{theorem}

\begin{proof}
Suppose on the contrary that $x^{*}\in X$ is a fixed point of $f$ and that $y\in X$ is a periodic point of prime period two. Thus,
\[
f(x^{*}) = x^{*},\qquad f(y)\neq y,\qquad f^{2}(y) = y.
\]
Since $y\neq f(y)$ and $x^{*}$ is fixed, the points $x^{*},y,f y$ are mutually distinct. Applying the perimeter contractive condition to these three points gives
\[
d(fx^{*},fy) + d(fy,f^{2}y) + d(f^{2}y,fx^{*})\leq q[d(x^{*},y) + d(y,fy) + d(fy,x^{*})].
\]
Using the identities $f x^{*} = x^{*}$ and $f^{2}y = y$, we obtain
\[
d(x^{*},fy) + d(fy,y) + d(y,x^{*})\leq q[d(x^{*},y) + d(y,fy) + d(fy,x^{*})].
\]
By the symmetry of the $b$-metric,
\[
d(y,x^{*}) = d(x^{*},y)\qquad \text{and}\qquad d(fy,y) = d(y,fy),
\]
and therefore
\[
P(x^{*},y,f y)\leq q P(x^{*},y,f y),
\]
where
\[
P(x^{*},y,f y) = d(x^{*},y) + d(y,f y) + d(fy,x^{*}).
\]
Since the points $x^{*},y,$ and $f y$ are mutually distinct, we have $P(x^{*},y,f y) > 0$. Dividing by this positive quantity yields $1\leq q$, which contradicts the assumption $0\leq q<1$. Hence, $f$ cannot possess both a fixed point and a periodic point of prime period two.
\end{proof}

\begin{corollary}
\label{cor:dichotomy}
If every periodic point of $f$ has prime period at most two, then either all periodic points of $f$ are fixed points or all periodic points belong to 2-cycles.
\end{corollary}




\begin{theorem}
\label{thm:even_iterates}
Let $(X,d)$ be a complete $b$-metric space with coefficient $s\geq 1$ and  $f:X\to X$  a contracting perimeters of triangles mapping with contraction constant $q\in [0,1)$. Suppose that $f$ possesses a periodic point of prime period two, and  $sq^{2}< 1$. Then there exists $n_0\in \mathbb{N}$ such that $f^{2n}$ is a continuous graphic contraction on $(X,d)$ for every $n\geq n_0$, with coefficient
\[
\lambda_n = \frac{s q^{2n}}{1 - s q^{2n}} < 1.
\]
\end{theorem}

\begin{proof}
Since $f$ possesses a periodic point of prime period two, it follows directly from Theorem~\ref{thm:no_fixed_with_2cycle} that $f$ has no fixed points on $X$.

Fix an even integer iteration index $m=2n$. We claim that the iterate $f^m$ constitutes a graphic contraction whenever $s q^m < 1$. We analyze the behavior of the orbital triple $(x, f^m x, f^{2m} x)$ by distinguishing between its geometric configurations. We now consider the following two cases

\medskip\noindent
\textbf{Case 1}: Suppose that the triple is not mutually distinct (Degenerate Configurations).
\begin{itemize}
    \item If $x = f^m x$, then $x$ is a periodic point whose period divides the even integer $m$. In this configuration, it follows immediately that $f^{2m}x = f^m(f^mx) = f^mx = x$. Thus, the operational step yields a trivial distance: $d(f^m x, f^{2m} x) = d(x, x) = 0$.
    \item If $x = f^{2m} x$, the first and third elements of the triple coincide. To evaluate the graphic contraction condition, we examine the target relation directly. Substituting $f^{2m}x = x$ yields the inequality $d(f^m x, x) \le \lambda_n d(x, f^m x)$. Because the coefficient satisfies $\lambda_n < 1$, this relation holds if and only if $d(x, f^m x) = 0$, which simplifies to $x = f^m x$, reducing back to a zero-distance state.
    \item If $f^m x = f^{2m} x$, the distance expression vanishes identically: $d(f^m x, f^{2m} x) = 0$.
\end{itemize}
Hence, all degenerate configurations trivially satisfy the graphic contraction condition.

\medskip\noindent
\textbf{Case 2}: Suppose that the triple is mutually distinct
When the elements $(x, f^m x, f^{2m} x)$ form a non-degenerate triangle, we apply the exact perimetric bounding and inductive infinite-series enveloping sequence detailed in the proof of Theorem~\ref{thm:main_nonperiodic}. Replacing the standard iteration index with our fixed even parameter $m$ yields
\[
d(f^m x, f^{2m} x) \le \frac{s q^m}{1 - s q^m} d(x, f^m x).
\]
Setting $m=2n$ and selecting an index $n \ge n_0$ sufficiently large such that $s q^{2n}<1$ guarantees that $\lambda_n < 1$, completing the verification of the graphic contraction property.

\medskip\noindent
The global continuity of the iterate $f^{2n}$ is guaranteed by Proposition~\ref{prop:continuity}. Since the underlying mapping $f$ is continuous on $X$, any finite composition of $f$ inherits this topological regularity. 
\end{proof}

\begin{example}
\label{ex:even_iterates}
(Adapted from \cite{cvetkovic2026} to $b$-metric spaces). Let
\(X = \{a,b,c\}\)
be equipped with the discrete metric
\[
d(x,y) = 
\begin{cases}
0, & x=y,\\
1, & x\neq y.
\end{cases}
\]
Then $(X,d)$ is a complete metric space and hence a complete $b$-metric space with coefficient $s=1$. Define $f:X\to X$ by
\[
f(a) = b,\qquad f(b) = a,\qquad f(c) = a.
\]
The points $a$ and $b$ form a periodic orbit of prime period two:
\[
f(a) = b,\qquad f(b) = a.
\]

\noindent\textbf{Verification of the perimeter contraction property.}
Since $X$ contains exactly three distinct points, it suffices to consider the triple $(a,b,c)$. Its perimeter equals
\[
P(a,b,c) = d(a,b) + d(b,c) + d(c,a) = 1 + 1 + 1 = 3.
\]
Moreover,
\[
(fa,fb,fc) = (b,a,a),
\]
and therefore
\[
P(fa,fb,fc) = d(b,a) + d(a,a) + d(a,b) = 1 + 0 + 1 = 2.
\]
Hence,
\[
P(fa,fb,fc) = \frac{2}{3} P(a,b,c).
\]
Consequently, $f$ contracts perimeters of triangles with contractive constant $q = \frac{2}{3} < 1$.

\noindent\textbf{Odd iterates are not graphic contractions.}
Let $m = 2n + 1$ be odd. Since
\[
f^{m}(a) = b,\qquad f^{2m}(a) = a,
\]
we obtain
\[
d(f^{m}a,f^{2m}a) = d(b,a) = 1,
\]
while
\[
d(a,f^{m}a) = d(a,b) = 1.
\]
If $f^{m}$ were a graphic contraction, there would exist $\lambda_{m}<1$ such that
\[
d(f^{m}a,f^{2m}a)\leq \lambda_{m}d(a,f^{m}a),
\]
which would imply $1\leq \lambda_{m}$, a contradiction. Therefore no odd iterate of $f$ is a graphic contraction.

\noindent\textbf{Even iterates are graphic contractions.}
Let $m = 2n$ be even. Then
\[
f^{m}(a) = a,\qquad f^{m}(b) = b.
\]
Moreover,
\[
f(c) = a,\qquad f^{2}(c) = b,\qquad f^{3}(c) = a,\qquad f^{4}(c) = b,
\]
and therefore $f^{2n}(c) = b$ for every $n\geq 1$. Consequently,
\[
f^{4n}(a) = a,\qquad f^{4n}(b) = b,\qquad f^{4n}(c) = b,
\]
and hence
\[
d(f^{2n}x,f^{4n}x) = 0
\]
for every $x\in X$. Therefore,
\[
d(f^{2n}x,f^{4n}x) = 0\leq 0\cdot d(x,f^{2n}x),\qquad x\in X.
\]
Thus every even iterate $f^{2n}$ is a graphic contraction with contractive constant $\lambda_{n} = 0$. Hence the conclusion of the even-iterate theorem holds with $n_0 = 1$.
\end{example}

\begin{corollary}
\label{cor:exactly_two_periodic}
Let $(X,d)$ be a complete $b$-metric space with coefficient $s\geq 1$, and $f:X\to X$ be a contracting perimeters of triangles mapping. If $f$ possesses a periodic point of prime period two, then $f$ has exactly two periodic points, both of which have prime period two.
\end{corollary}

\begin{proof}
Assume that $x^{*}\in X$ is a periodic point of prime period two. Then by definition,
\[
f^{2}(x^{*}) = x^{*},\qquad f(x^{*})\neq x^{*}.
\]
Consequently, the points $x^{*}$ and $f(x^{*})$ are distinct and form a periodic orbit of length two (a 2-cycle). Since
\[
f^{2}(f(x^{*})) = f(f^{2}(x^{*})) = f(x^{*}),
\]
and $f(x^{*})\neq f^{2}(x^{*})$, the point $f(x^{*})$ is also a periodic point of prime period two.

By Lemma \ref{lem:no_period_gt2}, $f$ cannot possess periodic points of a prime period strictly greater than two. Furthermore, Theorem \ref{thm:no_fixed_with_2cycle} establishes that a fixed point (a periodic point of prime period one) and a 2-cycle cannot coexist under an CPTM. Therefore, $f$ possesses no fixed points, meaning all potential periodic points must have a prime period of exactly two.

Assume on the contrary that there exists another periodic point $y\in X$ of prime period two such that
\[
y\notin \{x^{*},f(x^{*})\}.
\]
This assumption implies that the three points $x^{*},f(x^{*}),$ and $y$ are mutually distinct.

Since $y$ has prime period two, we have $f^{2}(y) = y$ and $f(y)\neq y$. We claim that $f(y)$ is also distinct from both $x^{*}$ and $f(x^{*})$. Indeed, if $f(y) = x^{*}$, applying $f$ to both sides yields $y = f^{2}(y) = f(x^{*})$, which contradicts $y\neq f(x^{*})$. Similarly, if $f(y) = f(x^{*})$, applying $f$ yields $y = f^{2}(y) = f^{2}(x^{*}) = x^{*}$, which contradicts $y\neq x^{*}$. Hence, the points
\[
f(x^{*}),\quad f(y),\quad x^{*}
\]
are also mutually distinct.

We can now properly apply the perimeter contraction condition to the triple of mutually distinct points $(x^{*},y,f(x^{*}))$, which yields:
\begin{equation}
P\big(f(x^{*}),f(y),f^{2}(x^{*})\big)\leq q P\big(x^{*},y,f(x^{*})\big).
\label{eq:cor_35}
\end{equation}
Using the identity $f^{2}(x^{*}) = x^{*}$, inequality \eqref{eq:cor_35} becomes:
\begin{equation}
P\big(f(x^{*}),f(y),x^{*}\big)\leq q P\big(x^{*},y,f(x^{*})\big).
\label{eq:cor_36}
\end{equation}
Analogously, because the points $f(x^{*})$, $f(y)$, and $x^{*}$ are mutually distinct, we can apply the perimeter contraction condition a second time, now to the triple $(f(x^{*}),f(y),x^{*})$:
\begin{equation}
P\big(f^{2}(x^{*}),f^{2}(y),f(x^{*})\big)\leq q P\big(f(x^{*}),f(y),x^{*}\big).
\label{eq:cor_37}
\end{equation}
Using the periodic identities $f^{2}(x^{*}) = x^{*}$ and $f^{2}(y) = y$, inequality \eqref{eq:cor_37} reduces to:
\begin{equation}
P\big(x^{*},y,f(x^{*})\big)\leq q P\big(f(x^{*}),f(y),x^{*}\big).
\label{eq:cor_38}
\end{equation}
Combining inequalities \eqref{eq:cor_36} and \eqref{eq:cor_38} yields:
\[
P\big(x^{*},y,f(x^{*})\big)\leq q\cdot P\big(f(x^{*}),f(y),x^{*}\big)\leq q^{2}P\big(x^{*},y,f(x^{*})\big).
\]
Because the points $x^{*},f(x^{*})$, and $y$ are mutually distinct, the axioms of the $b$-metric space guarantee that the distances between them are strictly positive, which implies that their perimeter expression is strictly positive:
\[
P\big(x^{*},y,f(x^{*})\big) = d(x^{*},y) + d(y,f(x^{*})) + d(f(x^{*}),x^{*}) > 0.
\]
Dividing by this positive quantity yields $1 \leq q^{2}$, which  contradicts the contractive assumption that $0\leq q<1$. Consequently, no periodic point of prime period two can exist outside of the set $\{x^{*},f(x^{*})\}$. It follows that $f$ possesses exactly two periodic points, both of which belong to the unique 2-cycle orbit.
\end{proof}

\section{Conclusion}
The main motivation for this work was the recent paper of Cvetković \cite{cvetkovic2026}, where it was shown that sufficiently large iterates of mappings contracting perimeters of triangles become graphic contractions in complete metric spaces.

In this paper, we have extended these results to the broader framework of complete $b$-metric spaces. The presence of the coefficient $s$ in the $b$-triangle inequality gives rise to additional technical difficulties that are absent in ordinary metric spaces requiring new estimates for the iterates of the mapping.

We proved that contracting perimeters of triangles mapping on complete $b$-metric spaces cannot possess periodic points of prime period greater than two. In the absence of periodic points of prime period two, sufficiently large iterates become graphic contractions, thereby extending the principal theorem of Cvetković to $b$-metric spaces.

Furthermore, we established that fixed points and periodic points of prime period two cannot coexist. When a two-cycle exists, we showed that sufficiently large even iterates are graphic contractions, providing an extension of the metric-space theory to the periodic setting.

Thus, our work generalizes several results of Cvetković from metric spaces to complete $b$-metric spaces and demonstrates that the relationship between perimeter contractions and graphic contractions remains valid in this more general setting.

Future research directions include the investigation of  contracting perimeters of polygons mappings in generalized metric spaces, multivalued versions of perimeter contractions, and extensions to partial metric spaces, rectangular $b$-metric spaces, modular spaces, and convex $b$-metric spaces.

\end{document}